\documentclass[11pt,reqno,oneside]{amsart}
 \usepackage{amsmath,amssymb,cite,epsfig,xspace,alltt,verbatim}
\numberwithin{equation}{section}

\theoremstyle{plain}
\newtheorem{theorem}{Theorem}[section]
\newtheorem{lemma}{Lemma}[section]

\theoremstyle{definition}

\theoremstyle{remark}
\newtheorem{remark}{Remark}[section]

\newcommand{\ud}{\mathrm{d}}

\newcommand{\Int}{\mathbb Z}
\newcommand{\oneb}{{\mathbf 1}}

\newcommand{\bb}{{\mathbf b}}
\newcommand{\fb}{{\mathbf f}}
\newcommand{\gb}{{\mathbf g}}
\newcommand{\rb}{{\mathbf r}}
\newcommand{\ub}{{\mathbf u}}
\newcommand{\xb}{{\mathbf x}}
\newcommand{\yb}{{\mathbf y}}
\newcommand{\zerob}{{\mathbf 0}}

\newcommand{\E}{{\mathcal E}}

\newcommand{\lpnorm}[2]{\left|\left|#1\right|\right|_{\ell^{#2}_h}}
\newcommand{\wpnorm}[2]{\left|\left|#1\right|\right|_{w^{1,#2}_h}}

\newcommand{\half}{{\textstyle \frac{1}{2}}}

\newcommand{\pd}[2]{\frac{\partial #1}{\partial #2}}

\begin{document}

\title{An Analysis of the Effect of Ghost Force Oscillation on Quasicontinuum Error}
\author{Matthew Dobson}
\author{Mitchell Luskin}

\address{Matthew Dobson\\
School of Mathematics \\
University of Minnesota \\
206 Church Street SE \\
Minneapolis, MN 55455 \\
U.S.A.}
\email{dobson@math.umn.edu}

\address{Mitchell Luskin \\
School of Mathematics \\
University of Minnesota \\
206 Church Street SE \\
Minneapolis, MN 55455 \\
U.S.A.}
\email{luskin@umn.edu}

\thanks{
This work was supported in part by DMS-0757355,
 DMS-0811039,  the Institute for Mathematics and
Its Applications,
 the University of Minnesota Supercomputing Institute, and
the University of Minnesota Doctoral Dissertation Fellowship.
  This work is also based on
work supported by the Department of Energy under Award Number
DE-FG02-05ER25706.
}

\keywords{} 

\subjclass[2000]{65Z05,70C20}

\date{\today}

\begin{abstract}
The atomistic to continuum interface for quasicontinuum energies
exhibits nonzero forces under uniform strain that have been
called ghost forces.
In this paper,
we prove for a linearization of a one-dimensional quasicontinuum energy
around a uniform strain
that the effect of the ghost forces on the displacement
nearly cancels and has a small effect on the error away from the interface.
We give optimal order error estimates
that show that the quasicontinuum displacement
converges to the atomistic displacement at the rate O($h$)
in the discrete $\ell^\infty$ and
$w^{1,1}$ norms where $h$ is the interatomic spacing.
We also give a proof that the error in the displacement gradient
decays away from the interface to O($h$) at distance O($h|\log h|$)
in the atomistic region and distance O($h$) in the continuum region.
E, Ming, and Yang previously gave a counterexample to convergence in the $w^{1,\infty}$ norm for a harmonic interatomic potential.
Our work gives an explicit and simplified form for the decay of the effect of the
atomistic to continuum coupling error in terms of a general underlying interatomic potential and gives the estimates described above in the discrete $\ell^\infty$ and $w^{1,p}$ norms.
\end{abstract}

\maketitle
{
\thispagestyle{empty}

\section{Introduction}

The quasicontinuum method (QC) reduces the computational complexity of
atomistic simulations by replacing smoothly varying regions of the
material with a continuum approximation derived from the atomistic
model~\cite{pinglin03, pinglin05,
legollqc05, ortnersuli, OrtnerSueli:2006d, e05, tadm96, knaportiz, e06,
mill02, rodney_gf, miller_indent, curtin_miller_coupling,
jacobsen04,dobs08}.  This is extremely effective in simulations
involving defects, which have singularities in the deformation gradient.
In such simulations, a few
localized regions require the accuracy and high computational expense
of atomistic scale resolution, but the rest of the material has a
slowly varying deformation gradient which can be more efficiently computed
using the continuum approximation without loss of the desired accuracy.
Adaptive algorithms have been developed for QC to determine
which regions require the accuracy of atomistic modeling and how to
coarsen the  finite element mesh in the continuum
region~\cite{ortnersuli,OrtnerSueli:2006d,prud06,oden06,arndtluskin07a,arndtluskin07b,ArndtLuskin:2007c}.
The atoms retained in the atomistic region and the atoms at nodes of
the piecewise linear finite element mesh in the continuum region are
collectively denoted as {\em representative atoms}.

Recent years have seen the development of many QC approximations that differ
in how they compute interactions among the representative atoms.
In the following, we concern
ourselves with the original energy-based quasicontinuum (QCE)
approximation~\cite{tadm96,mill02}, but the phenomena that we analyze
occur in all other quasicontinuum approximations, as well as
in other multiphysics coupling methods~\cite{e06}.  In QCE, a total energy is
created by summing energy contributions from each representative atom in
the atomistic region and from each element in the continuum region,
where the volume of the elements in the atomistic to continuum interface
is modified to exactly conserve mass.  This construction was
chosen so that for any uniform strain the QCE energy,
the continuum energy, and the atomistic energy are identical.
(As discussed later, this conservation property for the QCE approximation
is not sufficient to prevent the existence of nonzero forces at the
atomistic to continuum interface for uniform strain.)
The representative
atoms then interact via forces defined by the total energy.  This
makes for a simple and versatile method that can treat complicated
geometries and can be used with adaptive algorithms that modify the mesh and
atomistic regions during a quasi-static process.
Other atomistic to continuum approaches have been
proposed, for example, that utilize overlapping or blended
domains~\cite{BadiaParksBochevGunzburgerLehoucq:2007,ParksBochevLehoucq:2007}.

One drawback of the energy-based quasicontinuum approximation that has
received much attention is the fact that at the atomistic to continuum
 interface the balance of force equations do not give a
consistent scheme~\cite{shenoy_gf}.  As explained in Section~\ref{sec:model}, the
equilibrium equations in the interior of both the atomistic region and
the continuum region give consistent finite difference schemes for the
continuum limit, whereas the QC equilibrium equations near the
interface are not consistent with the continuum limit.  This is most
easily seen by considering a uniform strain, which
will be assigned identically zero elastic forces by any
consistent scheme.  (Ensuring that a given scheme assigns zero forces for
uniform strain has been known as the ``patch test'' in
the theory of finite elements~\cite{strangfix}.)
The nonzero residual forces present in QCE
for uniform strain
have been called ``ghost forces''~\cite{shenoy_gf,dobs08}.

In this paper, we give optimal order error estimates for the effect of the
inconsistency on the displacement and displacement gradient for a linearization of
a one-dimensional atomistic energy
and its quasicontinuum approximation.
We consider the linearization of general interatomic potentials
which are concave near second-neighbor interatomic distances.
This property guarantees that the interfacial error due to the Cauchy-Born
approximation with a second-neighbor cut-off is positive~\cite[p. 117]{dobs08} and 
that the quasicontinuum
error is not oscillatory in the atomistic region (see Section~\ref{sec:comp}).
Similar optimal order error estimates have been given by E, Ming, and Yang~\cite{emingyang}
for a harmonic interatomic potential.

We begin
by linearizing a one-dimensional atomistic energy, its local quasicontinuum approximation
(which we will call the continuum energy),
and its quasicontinuum approximation
about a uniform strain
for a second-neighbor atomistic energy.
We will show in Section~\ref{sec:model} that the three systems of equilibrium equations are then
\begin{equation*}
\begin{split}
L^{a,h}  \ub_a    &= \fb,\qquad \qquad \text{(atomistic)}\\
L^{c,h}  \ub_c    &= \fb,\qquad \qquad \text{(continuum)}\\
L^{qc,h} \ub_{qc}-\gb &= \fb, \qquad \qquad \text{(quasicontinuum)}
\end{split}
\end{equation*}
where $\fb$ is an external loading, $L$ and $\ub$ are the linearized operator
and corresponding displacement for each scheme,
$\gb$  is non-zero only in the atomistic
to continuum interface, and $h$ is the interatomic spacing.
The term $\gb$ in the quasicontinuum equilibrium
equations is due to the unbalanced second-neighbor
interactions in the interface~\eqref{ghost} and
for uniform stretches is precisely
the ghost force described in~\cite{shenoy_gf,dobs08,mill02}.


Formally, the error decomposes as
\begin{equation*}
\ub_a - \ub_{qc} = ( (L^{a,h})^{-1} - (L^{qc,h})^{-1}) \fb -
(L^{qc,h})^{-1} \gb.
\end{equation*}
(The operators are all translation invariant, so they only have solutions
up to the choice of an additive constant.)
In this paper, we focus on the second term, $(L^{qc})^{-1} \gb,$ which is the
error due to the inconsistency at the interface.  To do so, we consider the
case of no external field, $\fb=\zerob,$ which will make
$\ub_a = \zerob.$  For most applications of the quasicontinuum method,
the only external field is due to loads that are applied on the
boundary of the material, far from the atomistic to continuum
interface.

We showed in~\cite{dobs08}
that the ghost forces are oscillatory and sum to zero.
In this paper, we prove that the error in the displacement
gradient is O(1) at the interface and
decays away from the interface to O($h$) at distance O($h|\log h|$)
in the atomistic region and distance O($h$) in the continuum region.
As noted above, similar results have been given
in~\cite{emingyang} for a harmonic interatomic potential
with $\fb\ne \zerob$ and
Dirichlet boundary conditions.  Here, we present a simplified approach
starting from
a linearization of a quasicontinuum approximation with a concave
second-neighbor interaction.  We explicitly give the form of the solution
and analyze the solution in discrete $l^\infty$ and $w^{1,p}$ norms.
We 
show that the quasicontinuum displacement converges to the atomistic
displacement at the
rate O($h$) in the discrete $l^\infty$ and $w^{1,1}$ norms where $h$
is the interatomic spacing.

In Section~\ref{sec:model}, we describe the energy-based
quasicontinuum approximation (QCE) and set up the analysis.
In Section~\ref{sec:comp}, we prove
Theorem~\ref{thm2} for the quasicontinuum energy that 
gives an optimal order, O($h$) error estimate in the  $l^\infty$ norm and a
O($h^{1/p}$) error estimate in the $w^{1,p}$ norm for $1 \leq p <
\infty.$  Note that for simplicity
the models and analysis are presented for the
case where no degrees of freedom have been removed in the continuum region,
but
we explain in Remark~\ref{coarsen} that identical results hold when the continuum region is
coarsened.  We present numerical computations in 
Figure~\ref{qc_comp} that clearly show that the error is localized in
the atomistic to continuum interface.

\section{One-Dimensional, Linear Quasicontinuum Approximation}
\label{sec:model}
We consider an infinite one-dimensional chain of atoms with
periodicity $2F$ in the deformed configuration.
Let $y_j$
denote the atomic positions for $-\infty<j<\infty,$  where there
are $2N$ atoms in each period.    Let
$h = 1/N$ and let
\begin{equation*}
u_j := y_j -  F h j
\end{equation*}
denote the displacement from the average interatomic spacing, $F h.$
In the following, we analyze the behavior
of the quasicontinuum method as the atomistic chain approaches the
continuum limit with $F$ fixed and $N \rightarrow \infty.$

The atomistic energy for a period of the chain is
\begin{equation}\label{en}
\E^{tot,h}(\yb) := h \sum_{j=-N+1}^{N}
\left[\phi\left(\frac{y_{j+1} - y_{j}}{h} \right)
+ \phi\left(\frac{y_{j+2} - y_{j}}{h}\right) - f_j y_{j}\right],
\end{equation}
where  $\phi(r)$ is a two-body interatomic
potential (for example, the Lennard-Jones potential
${\phi}(r) = 1/ r^{12} - 2/r^6$) and $\fb = (f_{-N+1},\dots,f_{N})$ are
external forces applied as dead loads on the atoms.
The periodic conditions
\begin{equation*}
y_{j+2N} = y_j + 2 F \qquad\text{or}\qquad u_{j+2N} = u_j
\end{equation*}
allow $\E^{tot, h}$ to be written in terms of $\yb := (y_{-N+1},\dots,y_{N}).$
We assume that
$\sum_{-N+1}^N f_j=0,$ otherwise there are no energy minimizing solutions since
the elastic energy is translation invariant.  In the following, we discuss the
existence and uniqueness of solutions to each of the models we encounter.
We note that the energy per bond in~\eqref{en}
has been scaled like $h\phi(r/h).$  This scaling implies that if we
let $y_j = y(j/N)$ and $f_j = f(j/N)$ for $j=-N+1, \dots, N$ where
$y \in C^1([-1,1])$ and
$f \in C([-1,1]),$ then as $N\to \infty$ and $F$ is held fixed, the
energy of a period~\eqref{en} converges to
\begin{equation*}
\int^1_{-1} \phi(y'(x)) + \phi(2 y'(x)) - f(x) y(x) \, \ud x.
\end{equation*}

We expand first neighbor terms around $F,$ giving
\begin{equation*}
\begin{split}
\phi \left(\frac{y_{j+1} - y_{j}}{h}\right) &=
\phi \left(F + \frac{u_{j+1} - u_{j}}{h}\right) \\
&= \phi(F)
+ \phi'(F) \left.\frac{u_{j+1} - u_{j}}{h}\right.
+ \half \phi''(F) \left(\frac{u_{j+1} - u_{j}}{h}\right)^2
+O\left(\left|\frac{u_{j+1}-u_{j}}{h}\right|^3\right),
\end{split}
\end{equation*}
and the second neighbor terms around $2F,$ giving
\begin{equation*}
\begin{split}
\phi \left(\frac{y_{j+2} - y_{j}}{h}\right)
&= \phi \left(2F + \frac{u_{j+2} - u_{j}}{h}\right) \\
&= \phi(2 F) + \phi'(2 F) \left.\frac{u_{j+2} - u_{j}}{h}\right.
+ \half \phi''(2 F) \left(\frac{u_{j+2} - u_{j}}{h}\right)^2
+O\left(\left|\frac{u_{j+2}-u_{j}}{h}\right|^3\right).
\end{split}
\end{equation*}

\subsection{Atomistic Model}
The linearized atomistic energy is then given by
\begin{equation}\label{at}
\begin{split}
\E^{a,h}(\ub) &:= h \sum_{j=-N+1}^{N}
\left[ \phi'_F \left.\frac{u_{j+1} - u_{j}}{h}\right.
  + \half \phi''_{F} \left(\frac{u_{j+1} - u_{j}}{h}\right)^2 \right. \\
&\qquad\qquad \left. + \phi'_{2F} \left.\frac{u_{j+2} - u_{j}}{h}\right.
  + \half \phi''_{2 F} \left(\frac{u_{j+2} - u_{j}}{h}\right)^2
  - f_j u_j\right],
\end{split}
\end{equation}
where
$\phi'_F := \phi'(F),   \phi''_{F} := \phi''(F),
\phi'_{2F} := \phi'(2F),  \phi''_{2 F} := \phi''(2F),$
and $\ub := (u_{-N+1}, \dots, u_{N}).$  Note that here and in the following,
we neglect the additive constant
$\phi(F) + \phi(2F) - h \sum_{j=-N+1}^N f_j F h j$ in
the linearized energy.  We assume that $\phi \in C^2([r_0, \infty))$
for some $r_0$ such that $0< r_0 < F,$
and
\begin{equation}
\label{assume}
\phi''_F > 0 \text{ and } \phi''_{2F} < 0.
\end{equation}
This holds true for the Lennard-Jones potential for $F h$ below the load limit,
unless the chain is extremely compressed (less than 60\% of the equilibrium
length).  
The property $\phi''_{2F} <0$ ensures that the quasicontinuum
error is not oscillatory in the atomistic region (see Section~\ref{sec:comp}).

We furthermore assume that
\begin{equation}
\label{assume2}
\phi''_F + 5 \phi''_{2F} > 0,
\end{equation}
which will be sufficient to
give solutions to the QC equilibrium equations under the assumption of
no resultant external forces (see Lemma~\ref{wellposed}).  In contrast, the weaker assumption
$\phi''_F + 4 \phi''_{2F} >0$
is sufficient for the fully atomistic or fully continuum approximation.
The equilibrium equations,
$\frac1h\pd{\E^{a,h}}{u_j}(\ub) = 0,$ for the
atomistic model~\eqref{at} are
\begin{equation}\label{atom}
\begin{split}
(L^{a,h} \ub)_j
 &= \frac{- \phi''_{2 F} u_{j+2} - \phi''_{F} u_{j+1}
+ 2 (\phi''_{F} + \phi''_{2 F}) u_{j} - \phi''_{F} u_{j-1}
- \phi''_{2 F} u_{j-2}}{h^2} =  f_j, \\
&\hspace{1.9in} u_{j+2N} = u_j,
\end{split}
\end{equation}
for $-\infty < j < \infty.$ Note that scaling by $\frac{1}{h}$
makes this a consistent approximation of the boundary value problem
\begin{equation}
\label{bvp}
\begin{split}
- (\phi''_{F}+4\phi''_{2 F}) u''(x) = f &\qquad \textrm{ for } -\infty<x <\infty,\\
 u(x+2) = u(x) &\qquad \textrm{ for } -\infty<x <\infty.
\end{split}
\end{equation}
The linearized atomistic energy~\eqref{at} has a unique minimum (up to a
constant) if $\phi''_{F} + 4 \phi''_{2 F}>0,$ provided that
$\sum_{j=N-1}^{N} f_j = 0.$ Standard ODE results show that~\eqref{bvp} has a
unique solution (up to a constant)
provided that $\int^1_{-1} f(x) \, \ud x = 0 .$

\begin{remark}
\label{rem1}
For the atomistic energy~\eqref{at}, the linear terms sum to zero by the
periodicity of the displacement, since
\begin{equation*}
\begin{split}
 h \sum_{j=-N+1}^{N} &\left[ \phi'_F \ \frac{u_{j+1} - u_{j}}{h}
 + \phi'_{2F} \ \frac{u_{j+2} - u_{j}}{h} \right] \\
&=
\phi'_F \left[u_{N+1} - u_{-N+1}\right] +
\phi'_{2F} \left[u_{N+2} + u_{N+1} - u_{-N+2} - u_{-N+1}\right] = 0.
\end{split}
\end{equation*}
However, we keep these terms in the model since they do not sum to zero when
the atomistic model is coupled to the continuum approximation in the
quasicontinuum energy.  The resulting terms give a more accurate representation
of what happens in the non-linear quasicontinuum model.
\end{remark}

\subsection{Continuum Approximation}
The continuum approximation splits the chain into linear finite elements with
nodes given by the representative atoms, which we recall are a
subset of the atoms in the chain.  The energy of the chain is the sum
of element energies which depend only on the element's deformation
gradient, the linear deformation that interpolates its nodal
positions.  The energy of an element is then computed by applying the
element's deformation gradient to the reference lattice, computing the
energy per atom using the atomistic model, and multiplying by the
number of atoms in the element (where the boundary atoms are shared
equally between neighboring elements).  If the continuum approximation
is not coarsened (every atom is a representative atom), then
the continuum energy is given by
\begin{equation}\label{cont}
\begin{split}
\E^{c,h}(\ub) &:= h \sum_{j=-N+1}^{N} \left[ (\phi'_F + 2 \phi'_{2F})
\left(\frac{u_{j+1} - u_{j}}{h}\right)
+ \half (\phi''_F + 4 \phi''_{2F}) \left(\frac{u_{j+1} - u_{j}}{h}\right)^2
- f_j u_j\right].
\end{split}
\end{equation}
See~\cite{dobs08} for a derivation of the continuum energy and a
discussion of the error terms at the element boundaries.
For $j \in \{-N+1, \dots, N\}$, the equilibrium equations
for the continuum approximation are
\begin{equation}\label{cont1}
(L^{c,h} \ub)_j = (\phi''_F + 4 \phi''_{2F})
\left[\frac{- u_{j+1} + 2 u_{j} - u_{j-1}}{h^2} \right] =  f_j,
\end{equation}
which is also a consistent approximation for the boundary value
problem~\eqref{bvp}.
It is easy to see that the
continuum energy~\eqref{cont} has a unique minimum (up to a constant)
if $\phi''_{F} + 4 \phi''_{2 F}>0,$ provided that $\sum_{j=N-1}^{N} f_j = 0.$
The quasicontinuum method
inherently supports coarsening, but we neglect it here since in one dimension
this only changes the scaling of equilibrium equations.

\subsection{Splitting the Energy}
We can split the atomistic energy and the continuum energy into
per-atom contributions so that
\begin{equation*}
\E^{a,h}(\ub) = h \sum_{j=-N+1}^{N} \left[\E^{a,h}_j\left(\ub\right) - f_j u_j\right]
\quad \text{ and } \quad
\E^{c,h}(\ub) = h \sum_{j=-N+1}^{N} \left[\E^{c,h}_j\left(\ub\right) - f_j u_j\right].
\end{equation*}
There are many possible ways to define the per-atom contributions, and
we do this in such a
way that these contributions are linearizations of the ones in the fully
nonlinear case presented in~\cite{dobs08,tadm96}.  In this case, we split the
energy of each bond to obtain
\begin{equation}
\label{atomsplit}
\begin{split}
\E^{a,h}_j(\ub) :=
\frac{1}{2} \Bigg[ \phi'_F &\left.\frac{u_{j+1} - u_{j}}{h}\right.
  + \half \phi''_{F} \left(\frac{u_{j+1} - u_{j}}{h}\right)^2 \\
&+ \phi'_{2F} \left.\frac{u_{j+2} - u_{j}}{h}\right.
  + \half \phi''_{2 F} \left(\frac{u_{j+2} - u_{j}}{h}\right)^2 \Bigg]\\
+\frac{1}{2} \Bigg[ \phi'_F &\left.\frac{u_{j} - u_{j-1}}{h}\right.
  + \half \phi''_{F} \left(\frac{u_{j} - u_{j-1}}{h}\right)^2 \\
&+ \phi'_{2F} \left.\frac{u_{j} - u_{j-2}}{h}\right.
  + \half \phi''_{2 F} \left(\frac{u_{j} - u_{j-2}}{h}\right)^2 \Bigg],
\end{split}
\end{equation}
and
\begin{equation}
\label{contsplit}
\begin{split}
\E^{c,h}_j(\ub) := \frac{1}{2} &\left[ (\phi'_F + 2 \phi'_{2F})
\left(\frac{u_{j+1} - u_{j}}{h}\right)
+ \half (\phi''_F + 4 \phi''_{2F}) \left(\frac{u_{j+1} - u_{j}}{h}\right)^2 \right]\\
&+ \frac{1}{2} \left[ (\phi'_F + 2 \phi'_{2F})
\left(\frac{u_{j} - u_{j-1}}{h}\right)
+ \half (\phi''_F + 4 \phi''_{2F}) \left(\frac{u_{j} - u_{j-1}}{h}\right)^2 \right].
\end{split}
\end{equation}

\subsection{Energy-Based Quasicontinuum Approximation}
The energy-based quasicontinuum approximation partitions
the representative atoms into atomistic and continuum representative atoms and
assigns to each atom the split energy corresponding to its
type (\ref{atomsplit}-\ref{contsplit}).
We define the nodes $-N+1,\dots,-K-1$ and $K+1,\dots,N$ to be continuum and
$-K,\dots,K$ to be atomistic, where we assume that $2\le K\le N-2$
to ensure well-defined atomistic and continuum regions.
The quasicontinuum energy is then
\begin{equation}
\label{qceTot}
\E^{qc,h}(\ub) :=
\sum_{j = -N+1}^{-K-1}  \E_{j}^{c,h}\left(\ub\right) +
\sum_{j=-K}^{K}   \E_{j}^{a,h}\left(\ub\right) +
\sum_{j= K+1}^{N} \E_{j}^{c,h}\left(\ub\right) -
\sum_{j = -N+1}^{N} f_j u_j.
\end{equation}

Since the energy is quadratic, the equilibrium equations,
$\frac1h\pd{\E^{qc,h}}{u_j}(\ub_{qc})=0,$ take the form
\begin{equation}\label{qc1}
L^{qc,h} \ub_{qc}-\gb = \fb.
\end{equation}
For $0 \leq j \leq N,$ the QCE operator is  given by
\begin{equation*}
\label{fishtail}
\begin{split}
(L^{qc,h} \ub)_j &= \phi''_F \frac{-u_{j+1} +2 u_j - u_{j-1}}{h^2} \\
&+ \begin{cases}
\displaystyle
 4 \phi''_{2 F} \frac{-u_{j+2} +2 u_j - u_{j-2}}{4 h^2},
& 0 \leq j \leq K-2, \\[6pt]
\displaystyle
 4 \phi''_{2 F} \frac{-u_{j+2} +2 u_j - u_{j-2}}{4 h^2}
+ \frac{\phi''_{2 F}}{h} \frac{u_{j+2} - u_{j}}{2 h},  & j = K-1, \\[6pt]
\displaystyle
 4 \phi''_{2 F} \frac{-u_{j+2} +2 u_j - u_{j-2}}{4 h^2}
- \frac{2 \phi''_{2 F}}{h} \frac{u_{j+1} - u_{j}}{h}
+ \frac{\phi''_{2 F}}{h} \frac{u_{j+2} - u_{j}}{2 h},
& j = K, \\[6pt]
\displaystyle
4 \phi''_{2 F} \frac{-u_{j+1} +2 u_j - u_{j-1}}{h^2}
-  \frac{2 \phi''_{2 F}}{h} \frac{u_{j} - u_{j-1}}{h}
+ \frac{\phi''_{2 F}}{h} \frac{u_{j} - u_{j-2}}{2 h}, & j = K+1, \\[6pt]
\displaystyle
4 \phi''_{2 F} \frac{-u_{j+1} +2 u_j - u_{j-1}}{h^2}
+ \frac{\phi''_{2 F}}{h} \frac{u_{j} - u_{j-2}}{2 h}, & j = K+2, \\[6pt]
\displaystyle
4 \phi''_{2 F} \frac{-u_{j+1} +2 u_j - u_{j-1}}{h^2}, & K+3 \leq j
\leq N.
\end{cases}
\end{split}
\end{equation*}
Similarly, $\gb$ is given by
\begin{equation}
\label{ghost}
g_j = \begin{cases}
0, & 0 \leq j \leq K-2, \\
-\half \phi'_{2F} /h, & j = K-1, \\
\half \phi'_{2F} /h, & j = K, \\
\half \phi'_{2F} /h, & j = K+1, \\
-\half \phi'_{2F} /h, & j = K+2, \\
0, &  K+3 \leq j \leq N.\\
\end{cases}
\end{equation}
For space reasons, we only list the entries for $0\le j\le N.$  The equations
for all other $j\in\Int$ follow from symmetry and periodicity.
Due to the symmetry in the definition of the atomistic and continuum regions,
we have that $L^{qc,h}_{i,j} =  L^{qc,h}_{-i,-j}$  and
$g_{j} = -g_{-j}$ for $-N+1 \leq i,j \leq 0.$  To see this, we define
the involution operator $(S\ub)_j=-u_{-j}$ and observe that
$\E^{qc,h}(S\ub)=\E^{qc,h}(\ub).$
It then follows from the chain rule that
\[
S^TL^{qc,h} S\ub-S^T\gb-S^T\fb =L^{qc,h} \ub-\gb -\fb
\quad \text{for all periodic }\ub\text{ and }\fb.
\]
Since $S^T=S,$ we can
conclude that
\begin{equation}\label{S}
SL^{qc,h} S = L^{qc,h}\quad\text{and}\quad Sg=g.
\end{equation}
Note that the expression for $\gb$
does not depend on $\phi'_F$ since the first-neighbor terms
identically sum to zero in the energy~\eqref{qceTot}.
We can now observe that the QCE approximation~\eqref{qc1}
is not consistent with the continuum limit of the atomistic model~\eqref{bvp}.

The linear operator $L^{qc}$ has all uniform translations,
$\ub = c \oneb = (c, c, \dots, c),$ in its nullspace.
To see that this is the full nullspace, we consider the factored operator
$L^{qc} = D^T E^{qc} D,$ where $(D \ub)_j = \frac{u_{j+1} - u_j}{h}$ and
\begin{equation*}
(E^{qc} \rb)_j = \begin{cases}
 \phi''_{2 F} r_{j-1} + (\phi''_{F}+ 2 \phi''_{2 F}) r_{j} +
 \phi''_{2 F} r_{j+1}, & 0 \leq j \leq K-2, \\
 \phi''_{2 F} r_{j-1} + (\phi''_{F}+ \frac32 \phi''_{2 F}) r_{j} +
\frac12 \phi''_{2 F} r_{j+1}, & j =K-1, \\
\frac12 \phi''_{2 F} r_{j-1} + (\phi''_{F}+ 3 \phi''_{2 F}) r_{j} +
\frac12 \phi''_{2 F} r_{j+1}, & j = K, \\
\frac12 \phi''_{2 F} r_{j-1} + (\phi''_{F}+ \frac92 \phi''_{2 F}) r_{j},
& j = K+1, \\
(\phi''_{F}+ 4 \phi''_{2 F}) r_{j},
& K+2 \leq j \leq N.
\end{cases}
\end{equation*}
We see that $E^{qc}$ is diagonally dominant provided
$\phi''_{F} + 5 \phi''_{2F}>0,$ hence assumption~\eqref{assume2}
implies $E^{qc}$ is invertible.
So we have that the nullspace of
$L^{qc}$ is precisely the nullspace of $D.$
Thus, $L^{qc} \ub = \gb$ has a solution whenever $\sum_{j=-N+1}^N f_j = 0,$
since
$\sum_{j=-N+1}^N g_j = 0.$ This solution is unique up to
a constant.

We now gather together the existence and uniqueness results stated
for the models.
\begin{lemma}
\label{wellposed}
If $\sum_{j=-N+1}^N f_j = 0$ and $\phi''_F + 4 \phi''_{2F} > 0,$ then the
linearized atomistic energy~\eqref{at} and continuum approximation~\eqref{cont}
both have a global minimum that is unique up to an additive constant.

Under the slightly stronger assumption $\phi''_F + 5 \phi''_{2F} > 0,$ the quasicontinuum
energy~\eqref{qceTot} has a unique minimizer up to a constant.
\end{lemma}

Here, and in the following, we take $\fb = \zerob,$ in order to focus
on the effect of the ghost force $\gb.$ Under this assumption, we can
conclude that the unique mean zero solution to the QCE equilibrium
equations~\eqref{qc1} is odd.  This follows from $S^{-1}=S$ and
\eqref{S} which together imply that $S\ub$ is a solution if and only
if $\ub$ is.  Because $S$ preserves the mean zero property, we
conclude that $\ub_{qc}$ is odd.
The unique odd solution to the atomistic equations,
$L^{a,h} \ub_a  = \zerob,$ is $\ub_a =\zerob.$ Thus, the QCE
equilibrium equations,
\begin{equation}
\label{elqce}
L^{qc,h} \ub_{qc} - \gb=\zerob,
\end{equation}
are also the error equations, and the quasicontinuum solution is
the error in approximating $\ub_a.$

\subsection{Discrete Sobolev Norms}
The effect of the interface terms on the total error is
norm-dependent, so we now employ discrete analogs of Sobolev
norms~\cite{ortnersuli}.  We define the discrete weak derivative by
\begin{equation*}
u'_j = \frac{u_{j+1} - u_{j}}{h}.
\end{equation*}
For $1 \leq p < \infty$ the discrete Sobolev norms are given by
\begin{equation*}
\begin{split}
\lpnorm{u}{p} &= \left(\sum_{j=-N+1}^{N} h |u_j|^p\right)^{1/p},  \\
\wpnorm{u}{p} &= \lpnorm{u}{p}
               + \lpnorm{u'}{p},
\end{split}
\end{equation*}
and for $p = \infty$ by
\begin{equation*}
\begin{split}
\lpnorm{u}{\infty} &= \max_{-N+1 \leq j \leq N} |u_j|,  \\
\wpnorm{u}{\infty} &= \lpnorm{u}{\infty}
               + \lpnorm{u'}{\infty}.
\end{split}
\end{equation*}
The above discrete Sobolev norms are equivalent to the standard Sobolev
norms restricted to the continuous, piecewise linear interpolants
$u(x)$ satisfying $ u(j/N)=u_j$ for $j = -N+1,\dots,N.$

\section{Convergence of the Quasicontinuum Solution}
\label{sec:comp}

We now analyze the quasicontinuum error, $\ub_{qc}.$  We note that is
it theoretically possible to solve~\eqref{elqce} explicitly for $\ub_{qc};$
however, the form of the solution is complicated by the second-neighbor
coupling in the atomistic region, so we instead obtain estimates for the decay
of the error, $\ub_{qc},$ by analyzing a O($h^2$)-accurate approximation of the
error.
Figure~\ref{qc_comp} shows the results of solving~\eqref{elqce} numerically
for odd solutions, $u_j = -u_{-j},$ with three choices of lattice spacing
and two sets of parameters. Note that for both sets of parameters, the
magnitude decays linearly with $h,$ whereas the displacement gradient
is O($1$) in the atomistic to continuum region.
The following argument proves the
qualitative error behavior analytically.
\begin{figure}
\begin{center}
\includegraphics[height=7.5in]{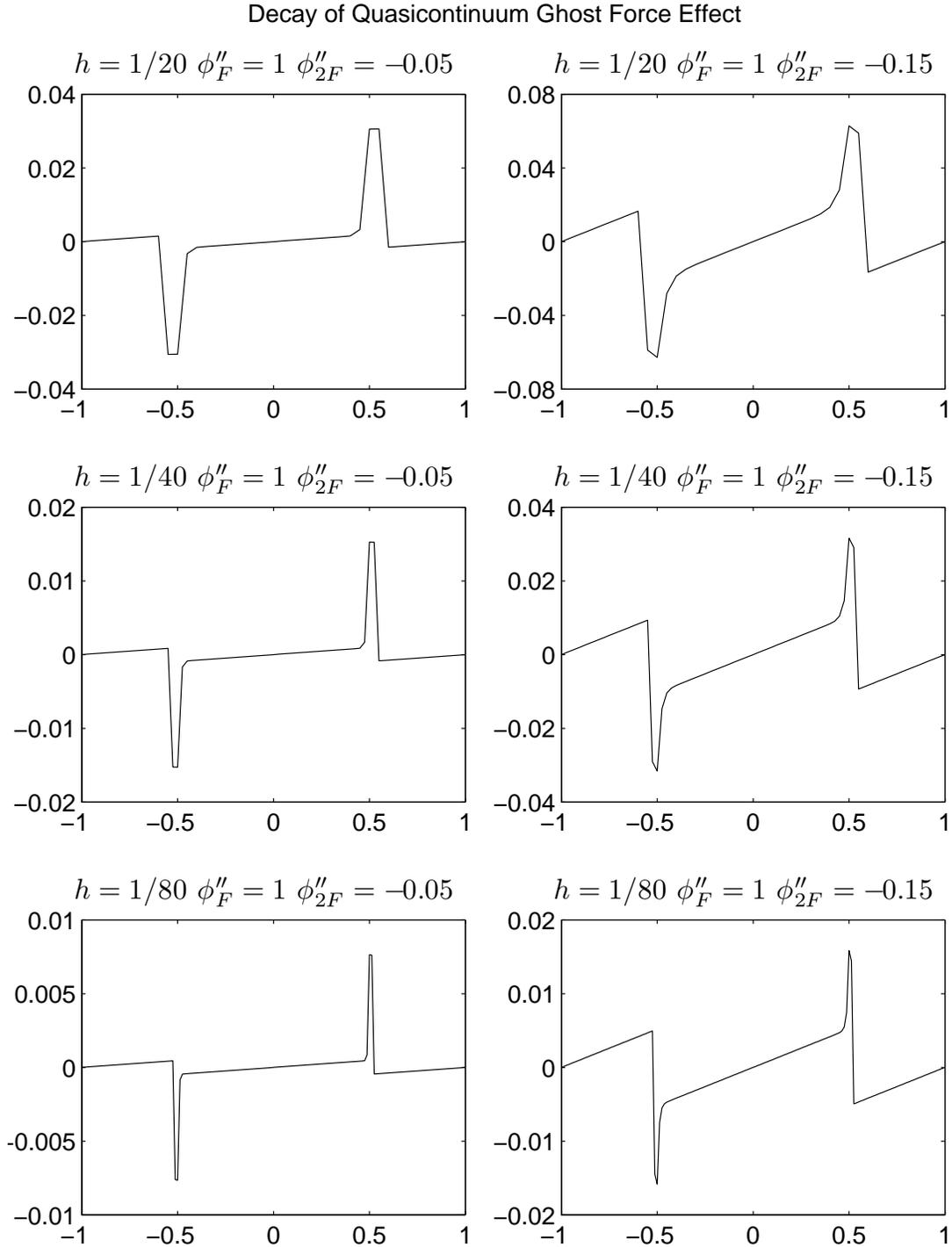}
\end{center}
\caption{\label{qc_comp}
Error for the energy-based quasicontinuum
approximation, $\ub_{qc}.$
We observe that the magnitude of the error is O($h$).
However, the oscillation near the interface means that the
error in the displacement gradient is O($1$)
in the interfacial region. The average deformation gradient, $F,$ for the right
column is close to failing the stability condition
$\phi''_{F} + 5 \phi''_{2 F} > 0.$ In all plots $K = N/2$ and
$\phi'_{2F} = 1.$}
\end{figure}

\subsection{Form of the Solution}
In the interior of the continuum region the solution is linear, but
in the atomistic region $\ub_{qc}$ is the sum of a linear
solution and
exponential solutions.
The homogeneous atomistic difference scheme
\begin{equation}\label{ato}
- \phi''_{2F} u_{j+2} - \phi''_F u_{j+1} + (2 \phi''_F + 2 \phi''_{2F}) u_j - \phi''_F u_{j-1} - \phi''_{2F} u_{j-2}
= 0
\end{equation}
has characteristic equation
\begin{equation*}
- \phi''_{2F} \Lambda^2 - \phi''_F \Lambda + (2 \phi''_F + 2 \phi''_{2F})
- \phi''_F \Lambda^{-1} - \phi''_{2F} \Lambda^{-2}
= 0,
\end{equation*}
with roots
\begin{equation*}
1, 1, \lambda, \frac{1}{\lambda},
\end{equation*}
where
\begin{equation*}
\lambda=\frac{(\phi''_F + 2 \phi''_{2F})
    + \sqrt{(\phi''_F)^2 + 4 \phi''_F \phi''_{2F}}}{-2 \phi''_{2F}}.
\end{equation*}
Based on the assumptions on $\phi$ in~\eqref{assume} and~\eqref{assume2}, we
have that $\lambda > 1.$  We note that if $\phi''_{2F}$ were positive
contrary to assumption~\eqref{assume}, then $\lambda$ would
be negative which would give a damped oscillatory error in the
atomistic region.  General solutions of the homogeneous atomistic
equations~\eqref{ato} have the form $u_j = C_1 + C_2 h j + C_3 \lambda^j + C_4
\lambda^{-j},$ but seeking an odd solution reduces this to the form
$u_j = C_2 hj + C_3 (\lambda^j - \lambda^{-j}).$

The odd solution of the quasicontinuum error equations~\eqref{elqce}
is thus of the form
\begin{equation*}
(\ub_{qc})_j = \begin{cases}
m_1 hj + \beta \left(\frac{\lambda^j - \lambda^{-j}}{\lambda^K}\right),
& 0 \leq j \leq K, \\
m_2 hj - m_2 + \tilde{u}_{K+1} , & j = K+1, \\
m_2 hj -m_2, & K+2 \leq j \leq N,
\end{cases}
\end{equation*}
where expressing the unknown $u_{K+1}$ using a perturbation of the
linear solution, $\tilde{u}_{K+1},$ simplifies the solution
of the equilibrium equations.
The four coefficients $m_1,\ m_2, \tilde{u}_{K+1}, \text{ and } \beta$ can be
found by satisfying the four equations in the interface, $j = K-1,\dots,K+2.$
Summing the equilibrium equations across the interface gives
\begin{equation*}
\begin{split}
0 &= \sum_{j=K-1}^{K+2} g_j = \sum_{j=K-1}^{K+2} (L^{qc,h} \ub_{qc})_j \\
  &= \phi''_F \left. \frac{u_{K-1} - u_{K-2}}{h^2} \right.
     + 4 \phi''_{2F} \left.
	          \frac{u_K + u_{K-1} - u_{K-2} - u_{K-3}}{4 h^2} \right.\\
  &  \qquad - (\phi''_F + 4 \phi''_{2F}) \left( \frac{u_{K+3} - u_{K+2}}{h^2} \right)\\
  &= (\phi''_F + 4 \phi''_{2F}) \left(\frac{m_1}{h} - \frac{m_2}{h}\right).
\end{split}
\end{equation*}
The cancellation of the exponential terms in the final equality holds
because
\begin{equation*}
\phi''_{2F}(\lambda^{K}-\lambda^{-K})+(\phi''_F+\phi''_{2F})(\lambda^{K-1}-\lambda^{-K+1}
-\lambda^{K-2}+\lambda^{-K+2})+\phi''_{2F}(-\lambda^{K-3}+\lambda^{-K+3})=0,
\end{equation*}
which can be seen by summing \eqref{ato} with the homogeneous solution
$u_j=-\lambda^j$ for $j=-K+2,\dots,K-2.$
Thus $m_1=m_2,$ that is, the slope of the linear part
does not change across the interface.
Hence, the odd solution is given by
\begin{equation}
\label{qcform}
(\ub_{qc})_j = \begin{cases}
m hj + \beta \left(\frac{\lambda^j - \lambda^{-j}}{\lambda^K}\right),
   & 0 \leq j \leq K, \\
m hj - m + \tilde{u}_{K+1} , & j = K+1, \\
m hj -m, & K+2 \leq j \leq N,
\end{cases}
\end{equation}
where the coefficients $m, \tilde{u}_{K+1}, \text{ and } \beta$ can now be found
by satisfying any three of the equations in the interface, $j = K-1,\dots,K+2.$

\subsection{Magnitude of the Solution}
We focus on the equations at $j = K-1, K+1, \text{ and } K+2$ and split the
interface equations as $(A_K + h B) \xb = h \bb,$ where
\begin{equation*}
\begin{split}
A_K &= \left[
\begin{array}{rrr}
 \frac{1}{2} \phi''_{2F}
  & -\frac{1}{2} \phi''_{2F}
  &\phi''_{2F} \gamma_{K+1} - \frac{1}{2} \phi''_{2F} \gamma_{K-1} \\[3pt]
-\phi''_F -\frac{5}{2} \phi''_{2F}
  &2\phi''_F +\frac{13}{2}\phi''_{2F}
  &-\phi''_F\gamma_{K} -2\phi''_{2F} \gamma_{K} - \frac{1}{2} \phi''_{2F} \gamma_{K-1} \\[3pt]
- \frac{1}{2} \phi''_{2F}
  & -\phi''_F - 4 \phi''_{2F}
  &- \frac{1}{2} \phi''_{2F} \gamma_{K}
\end{array}
\right], \\
B &= \left[
\begin{array}{rrr}
 \phi''_{2F} & 0 & 0 \\[3pt]
-\phi''_{2F} & 0 & 0\\[3pt]
 \phi''_{2F} & 0 & 0
\end{array}
\right], \quad
\xb =
\left[
\begin{array}{r}
m \\
\tilde{u}_{K+1} \\
\beta
\end{array}
\right], \quad
\bb = \frac{1}{2} \phi'_{2F}  \left[
\begin{array}{r}
-1 \\
 1 \\
-1
\end{array}
\right],
\end{split}
\end{equation*}
and $\gamma_j = \frac{\lambda^j - \lambda^{-j}}{\lambda^{K}}.$
We note that $A_K,$ $B,$ and $\bb$ do not depend on $h$ directly,
though $A_K$ may have indirect dependence if $K$ scales with $h$ as in
Figure~\ref{qc_comp}.
Therefore, we can neglect $B$ and conclude that $\xb$ is $O(h)$ provided
that $A_K^{-1}$ exists and is bounded uniformly in $K.$

\begin{lemma}
\label{fullrank}
For all $K$ satisfying $2\le K\le N-2,$ the matrix $A_K$ is nonsingular and
$||A_K^{-1}|| \leq C$ where $C>0$ is
independent of $K.$
\end{lemma}

\begin{proof}
Applying row reductions gives the upper triangular form
\begin{equation*}
\begin{split}
\widetilde{A} &= \left[
\begin{array}{rrr}
\frac{1}{2} \phi''_{2F} & -\frac{1}{2} \phi''_{2F} & \phi''_{2F} \gamma_{K+1} -
\frac{1}{2} \phi''_{2F} \gamma_{K-1} \\[3pt]
0       & -\phi''_F - \frac{9}{2} \phi''_{2F} & \phi''_{2F} \gamma_{K+1} -
		  \frac{1}{2} \phi''_{2F} \gamma_K  - \frac{1}{2} \phi''_{2F}
\gamma_{K-1}\\[3pt]
0       &  0 & \eta_K
\end{array}
\right]
\end{split}
\end{equation*}
where
\begin{equation*}
\eta_K = \textstyle  \left(
(\phi''_F)^2 + \frac{15}{2} \phi''_F \phi''_{2F} + \frac{53}{4}(\phi''_{2F})^2
      \right) \left(2 \gamma_{K+1} - \gamma_K - \gamma_{K-1}\right)
   + \frac{1}{2} \phi''_{2F} \left(\phi''_F  + \frac{9}{2}
   \phi''_{2F}\right) \left(\gamma_K - \gamma_{K-1}\right).
\end{equation*}
If the diagonal entries of $\widetilde{A}$ are
non-zero, then $A_K$ is nonsingular.
The coercivity assumption $\phi''_F + 5 \phi''_{2F} >0$~\eqref{assume2}
implies that
$-\phi''_F - 9/2 \phi''_{2F} < 0$ since $\phi''_{2F} <0,$ so the first
and second
diagonal entries are non-zero. Since the second term of $\eta_K$ is negative,
we can use the fact that
$\gamma_K-\gamma_{K-1} < 2 \gamma_{K+1} - \gamma_K - \gamma_{K-1}$ to see that
\begin{equation*}
\begin{split}
\eta_K &> \textstyle  \left(
   (\phi''_F)^2 + 8 \phi''_F \phi''_{2F} + \frac{62}{4} (\phi''_{2F})^2\right)
 \left(2 \gamma_{K+1} - \gamma_K - \gamma_{K-1}\right)\\
  &= \textstyle \left(\phi''_F
           + \left(4 + \frac{1}{\sqrt{2}}\right) \phi''_{2F}\right)
	  \left(\phi''_F + \left(4 - \frac{1}{\sqrt{2}}\right) \phi''_{2F}\right)
 \left(2 \gamma_{K+1} - \gamma_K - \gamma_{K-1}\right)\\
&> 0.
\end{split}
\end{equation*}
Therefore, $A^{-1}_K$ exists for all $K.$  Taking limits, we find
\begin{equation*}
\lim_{K\to\infty} \eta_K \geq
\textstyle \left(\phi''_F
           + \left(4 + \frac{1}{\sqrt{2}}\right) \phi''_{2F}\right)
	  \left(\phi''_F + \left(4 - \frac{1}{\sqrt{2}}\right) \phi''_{2F}\right)
 \left(2 \lambda - 1  - \lambda^{-1}\right) > 0,
\end{equation*}
where we note that the elementary matrices corresponding to the row reduction
operations did not depend on $K$ so that $\lim_{K\to\infty} A_K$ is
nonsingular.
The inverse of a matrix
is continuous as a function of the entries whenever the matrix is nonsingular.
Thus, the fact that $\lim_{K\to\infty} A_K$ is nonsingular implies that
$\lim_{K\to\infty} ||A_K^{-1}||$ is finite.
Since $||A_K^{-1}||$ is finite for all $K$ and
$\lim_{K\to\infty} ||A_K^{-1}||$ is finite, we conclude that
$||A_K^{-1}||$ is uniformly bounded.
\end{proof}

Thus, we have that $m, \tilde{u}_{K+1},$ and $\beta$ are all O($h$).  We can
express the derivative, $\ub'_{qc},$ as
\begin{equation*}
(\ub'_{qc})_j = \begin{cases}
m + \frac{\beta}{h}
  \left( \frac{\lambda^{j+1} - \lambda^{-j-1}}{\lambda^K }
  - \frac{\lambda^j - \lambda^{-j}}{\lambda^K } \right), & 0 \leq j \leq K-1, \\
m - \frac{m}{h} + \frac{\tilde{u}_{K+1}}{h}  - \frac{\beta}{h}
    \left( \frac{\lambda^K - \lambda^{-K}}{\lambda^K }\right), & j = K, \\
m - \frac{\tilde{u}_{K+1}}{h},                                 & j = K+1, \\
m,                                                        & K+2\leq j \leq N-1,
\end{cases}
\end{equation*}
where $u'_{-j-1} = u'_j$ for $j = 0, \dots, N-1.$

\begin{theorem}
\label{thm2}
Let $\ub_{qc}$ be the solution to the QC error equation~\eqref{elqce}.
Then for $1 \leq p \leq \infty, $ $2\le K\le N-2,$ and $h$ sufficiently small,
the error can be bounded by
\begin{equation*}
\begin{split}
\lpnorm{\ub_{qc}}{\infty} &\leq C h, \\
\wpnorm{\ub_{qc}}{p} &\leq C h^{1/p},
\end{split}
\end{equation*}
where $C>0$ is independent of $h, K,$ and $p.$
\end{theorem}

\begin{proof}
The result for the $\ell^{\infty}$ norm follows from the fact that
all terms in $\eqref{qcform}$ are O($h$).  To show the bound on
$w^{1,p},$ we first apply the triangle inequality to separate the $m,
\frac{\tilde{u}_{k+1}}{h}, \frac{m}{h},$ and $\frac{\beta}{h}$ terms
which we bound using the fact that $\tilde{u}_{K+1}, m,$ and $\beta$
are $O(h).$ We have
\begin{equation*}
\begin{split}
\wpnorm{\ub_{qc}}{p} &= \lpnorm{\ub_{qc}}{p} + \lpnorm{\ub_{qc}'}{p} \\
&\leq \lpnorm{\ub_{qc}}{p} + |m|
  + \left(2 \left| \frac{m}{h}\right|^{p} h \right)^{1/p}
  + \left(4 \left| \frac{\tilde{u}_{K+1}}{h}\right|^{p} h \right)^{1/p}
   \\
&\qquad + 2 \left( h \sum^{K}_{j=-K} \left|\frac{\beta}{h}
                    \frac{(\lambda^{j} - \lambda^{-j})}{\lambda^K}\right|^p
          \right)^{1/p} \\
&\leq C h^{1/p} + \frac{2 |\beta|}{h} \left( 2h \sum^{K}_{j=0}
           \left| \frac{\lambda^j }{\lambda^K} \right|^p
                    \right)^{1/p}\\
&\leq C h^{1/p} + \frac{2 |\beta|}{h} \left( 2 h
  \frac{ \lambda^p}{\lambda^p - 1} \right)^{1/p} \\
&\leq C h^{1/p}. \qedhere
\end{split}
\end{equation*}
\end{proof}

Finally, we show that the pointwise error in the derivative, $\ub'_{qc},$
decays exponentially in $j$ to O($h$) away from the interface in the atomistic
region and decays immediately to O($h$) away from the interface in the continuum region.
\begin{lemma}
There is a $C > 0$ such that
$|(\ub'_{qc})_j| \le C h$
for all $0\le j \leq K + \frac{\ln h}{\ln \lambda}$ and
$K+2 \leq j \leq N.$
Thus, the interface has size O($h |\log h|$).
\end{lemma}

\begin{proof}
For $h$ sufficiently small, we have that
$\max(m, \beta) \leq C h.$
Since $u'_j = m$ for $K+2 \leq j \leq N,$ in this region
$u'_j \le C h.$ For the terms $0 \leq j \leq K-1$ it is sufficient to show that
the exponential term is less than or equal to $C h.$
For $0 \leq j \leq K + \frac{\ln h}{\ln \lambda},$ we have that
\begin{equation*}
\begin{split}
\left( \frac{\lambda^{j+1} - \lambda^{-j-1}}{\lambda^K }
  - \frac{\lambda^j - \lambda^{-j}}{\lambda^K } \right)
& \leq \lambda^{j+1 - K} \\
&\leq \lambda^{K + \frac{\ln h}{\ln \lambda} +1 - K}\\
&\leq C h. \qedhere
\end{split}
\end{equation*}
\end{proof}

\begin{remark}\label{coarsen}
In order reduce the degrees of freedom, the continuum region is
coarsened in computations using the quasicontinuum method.
For simplicity, coarsening was omitted
from the model presented in this paper, but, in fact, the results are unchanged if it is used.
Conventionally, coarsening only occurs away from the atomistic to
continuum interface, so that no degrees of freedom are removed if they
interact directly with the atomistic region.  Since the solution $u_j$
is linear for $K+2 \leq j \leq N,$ any level of coarsening produces
an identical solution.

\end{remark}

}

\end{document}